\documentclass[10pt]{amsart}

\usepackage{amssymb}
\usepackage{graphicx}
\usepackage{enumerate}
\usepackage{multirow}
\usepackage{amsmath,color}
\usepackage{hyperref}
\usepackage{url}
\usepackage{mleftright}

\newtheorem{thm}{Theorem}[section]
\newtheorem{cor}[thm]{Corollary}
\newtheorem{lem}[thm]{Lemma}

\numberwithin{equation}{section}

\title[$p$-Domination Number and Multiplicity of Adjacency Eigenvalues]{A bound for the $p$-domination number of a graph in terms of its eigenvalue multiplicities}

\author{A. Abiad}
\address{Department of Mathematics and Computer Science, Eindhoven University of Technology, the Netherlands}
\address{Department of Mathematics: Analysis, Logic and Discrete Mathematics, Ghent University, Belgium}
\address{Department of Mathematics and Data Science, Vrije Universiteit Brussel, Belgium}
\email{a.abiad.monge@tue.nl}
\author{S. Akbari}
\address{Department of
Mathematical Sciences, Sharif University of Technology, Iran}
\email{s\_akbari@sharif.edu}
\author{M.H. Fakharan}
\address{Department of
Mathematical Sciences, Sharif University of Technology, Iran}
\email{mh.fakharan92@student.sharif.ir}
\author{A. Mehdizadeh}
\address{Department of
Mathematical Sciences, Sharif University of Technology, Iran}
\email{mehdizadehalireza7@gmail.com}
\date{}
\begin{document}


\begin{abstract}
Let $G$ be a connected graph of order $n$ with domination number $\gamma(G)$.  Wang, Yan, Fang, Geng and Tian [Linear Algebra Appl. 607 (2020), 307-318] showed that 
for any Laplacian eigenvalue $\lambda$ of $G$ with multiplicity $m_G(\lambda)$, it holds that $\gamma(G)\leq n-m_G(\lambda)$. Using techniques from the theory of star sets, in this work we prove that the same bound holds when $\lambda$ is an arbitrary adjacency eigenvalue of a non-regular graph, and we characterize the cases of equality. Moreover, we show a result that gives a relationship between start sets and the $p$-domination number, and we apply it to extend the aforementioned spectral bound to the $p$-domination number using the adjacency and Laplacian eigenvalue multiplicities.
\end{abstract}

\maketitle

\section{Introduction}
 A set $S \subseteq V(G)$ is called {\it dominating} ({\it total dominating}) if every $v \in V(G) \setminus S$ ($v\in V(G)$) is adjacent to some vertex in $S$. The {\it domination number} ({\it total domination number}) $\gamma(G)$ ($\gamma_t(G)$) is the minimum size of a dominating set of $G$ (total dominating set of $G$). For instance $\gamma(K_n)=\gamma(K_{1,n})=1$ and $\gamma_t( K_n)=\gamma(K_{1,n})=2$. Observe that $\gamma(G)\leq \gamma_t(G)$. More generally, a set $S$ is called $p$-\emph{dominating} if every $v\in V \backslash S$ is adjacent to at least $p$ vertices of $S$. The $p$-{\it domination number}, denoted by $\gamma_p(G)$, is the minimum size of a $p$-dominating set of $G$. 

The relation between the Laplacian eigenvalues of a graph and the domination number has received a great deal of attention in the literature. Bounds on the domination number involving the largest Laplacian eigenvalue are shown, among others, by Brand and Seifter \cite{bs}, Xing and Zhou \cite{xz}, Nikiforov \cite{N}. Bounds for the domination number using the second largest Laplacian eigenvalue are shown by Aouchiche, Hansen and Stevanović \cite{AHS}, and by Har \cite{H}. The domination number has also been studied in relation to the Laplacian eigenvalue distribution, see Hedetniem, Jacobs and Trevisan \cite{hjt}. The Laplacian eigenvalues have also been used to provide bounds for the $p$-domination number, see Abiad, Fiol, Haemers and Perarnau \cite{afhp}.  While several results are known to connect the domination number with the Laplacian eigenvalues, not much is known about the relation of the domination number with the adjacency spectrum for non-regular graphs. This work provides a step further in this direction.

Wang, Yan, Fang, Geng and Tian \cite[Theorem 4.5]{wyfgt} recently showed that if one considers $\lambda$ to be a Laplacian eigenvalue of a graph $G$ with multiplicity $m_G(\lambda)$, then it holds that $\gamma(G)\leq n-m_G(\lambda)$. While for regular graphs such bound also applies to the adjacency eigenvalues, it was not clear if that was the case for general graphs. Using a completely different approach based on techniques from star sets, in this work we show that such bound is also valid if one uses the adjacency eigenvalues of a graph, and we also study the tightness of our bound. Star sets were first introduced by Cvetković, Rowlinson and Simić in 1993 as a way to study eigenspaces of graphs and also to investigate the graph isomorphism problem  \cite{cvet3}. They soon became a powerful tool due to their strong link between graphs and linear algebra. This connection is promising in that it not only reflects the geometry of eigenspaces but also extends to combinatorial aspects. We extend results of Cvetković, Rowlinson and Simić, who studied the link between star sets and dominating sets. In particular, we show a new relation with $p$-dominating sets, and we use it to prove a spectral bound for $\gamma_p$.


\section{Preliminaries}\label{sec:preliminaries}

Let $G = (V(G),E(G))$ be a graph, where $V (G) = \{v_1, \ldots , v_n\}$ is the vertex set and $E(G)$ is the edge set of $G$, respectively. Throughout this paper all graphs are connected, simple and undirected. Let $d_G(v_i)$ be the degree of $v_i$ in $G$. Let $\delta(G)$ denote the minimum degree of  $G$.  Let $K_n$, $K_{r,s}$ and $C_n$ denote the complete graph of order $n$, the complete bipartite graph with part sizes $r$ and $s$, and the cycle of order $n$, respectively.

The {\em adjacency matrix} of $G$, denoted by $A(G)$, is an $n\times n$ matrix whose $(i, j)$-entry is $1$ if $v_i$ and $v_j$ are adjacent and $0$, otherwise. The eigenvalues of $A(G)$ are called the \emph{adjacency eigenvalues} (\emph{eigenvalues} for short) of $G$. The multiplicity of an eigenvalue $\lambda$ in a graph $G$ is denoted by $m_G(\lambda)$. The rank of a graph $G$ of order $n$ is $n-m_G(0)$ and is denoted by $rank(G)$.

A subset $X$ of $V (G) = \{v_1, \ldots , v_n\}$ is called a \emph{star set} if the matrix obtained from $A(G)$ by removing rows and columns corresponding to $X$ does not have $\lambda$ as an eigenvalue. In graph theory context, a star set for an eigenvalue $\lambda$ of $G$ is a subset $X$ of vertices such that $\lambda$ is not an eigenvalue of $G \setminus X$. The set $\bar{X}=G \setminus X$ is called a \emph{star complement} for $\lambda$ in $G$.

Next we recall two important results about star sets from \cite{R1994,cvet} which will be the key ingredients for our proofs.

\begin{lem}\label{dom}
{\em \cite[Proposition ~5.1.4]{cvet}.} Let $X$ be a star set for $\lambda$ in $G$, and let $\bar{X} = V(G) \setminus X$.\\
($i$) If $\lambda \neq 0$, then $\bar{X}$ is a dominating set for $G$.\\
($ii$) If $\lambda \neq -1$ or $0$, then $\bar{X}$ is a location-dominating set for $G$, that is the
$\bar{X}$-neighbourhoods of distinct vertices in $X$ are distinct and non-empty.
\end{lem}

 \begin{lem}\label{con}
{\em \cite[Theorem.~5.1.6]{cvet}}. Let $\lambda$ be an eigenvalue of a connected graph $G$, and let $K$ be a connected induced subgraph of $G$ not having $\lambda$ as an eigenvalue. Then $G$ has a connected star complement for $\lambda$ containing $K$.
\end{lem}

For more background information on star partitions, we refer the reader to \cite{cvet}.

\section{Multiplicity of the adjacency eigenvalues}\label{sec:results}

Our first result provides a sharp upper bound for the domination number and for the total domination number of $G$ in terms of the order of $G$ and the multiplicity of its adjacency eigenvalues. 
In order to show it, we will use a characterization of graphs having $rank(G)=2$ or $rank(G)=3$ which appears, among others, in \cite[Theorem 2]{chang}.



\begin{thm}\label{thm:first}
Let $G$ be a connected graph of order $n$, and let $\lambda$ be an eigenvalue of $G$ with multiplicity $m_G(\lambda)$. Then $$\gamma(G) \leq n-m_G(\lambda)$$ and  $\gamma(G)=n-m_G(\lambda)$ if and only if $$(\lambda,G)\in \{(0,K_1),(1,K_2),(-1,K_n),(0,K_{r,s}):\, r,s\geq 2\}.$$

Moreover if $G\neq K_n$, then $\gamma_t(G)\leq n-m_G(\lambda)$.
\end{thm}

\begin{proof}
If $G=K_n$, then $\gamma(G) =1$ and so $\gamma(G) \leq n-m_G(\lambda)$ for any eigenvalue $\lambda$ of $G$ with multiplicity $m_G(\lambda)$. If $G\neq K_n$, by Lemma \ref{con}, there exists a star set $X$ for $\lambda$ such that the induced subgraph on $\bar{X}$, say $H$, is connected. If $\lambda\neq 0$, then by  Lemma \ref{dom}$(i)$, $V(H)$ is a dominating set for $G$. If $\lambda=0$, then again $V(H)$ is a dominating set for $G$ because if $x\notin V(H)$, $N_H(x)=\emptyset$ and $y\in N(x)$, then the rank of the induced subgraph on the set $V(H)\cup \{x,y\}$ is larger than $rank(G)$, a contradiction. So $\gamma(G)\leq \gamma_t(G) \leq n-m_G(\lambda)$. \\
Suppose now that $\gamma(G) = n-m_G(\lambda)$. For the sake of notation simplicity, we denote $\gamma(G)=\gamma$. Assume that $V(H)=\{v_1,  \ldots, v_{\gamma} \}$. Then, for every $i=1, \ldots, \gamma$, the set $V(H)\setminus \{v_i\}$ is not a dominating set. So there exists at least one vertex $x_i \in V(G)\setminus V(H)$ such that $N_H(x_i)=\{v_i\}$. Suppose that $L$ is the induced subgraph on $\{x_1, \ldots\, ,x_{\gamma}\}$, see Figure \ref{tatabogh}.
\begin{figure}[h]
\centering
\includegraphics[width=40mm]{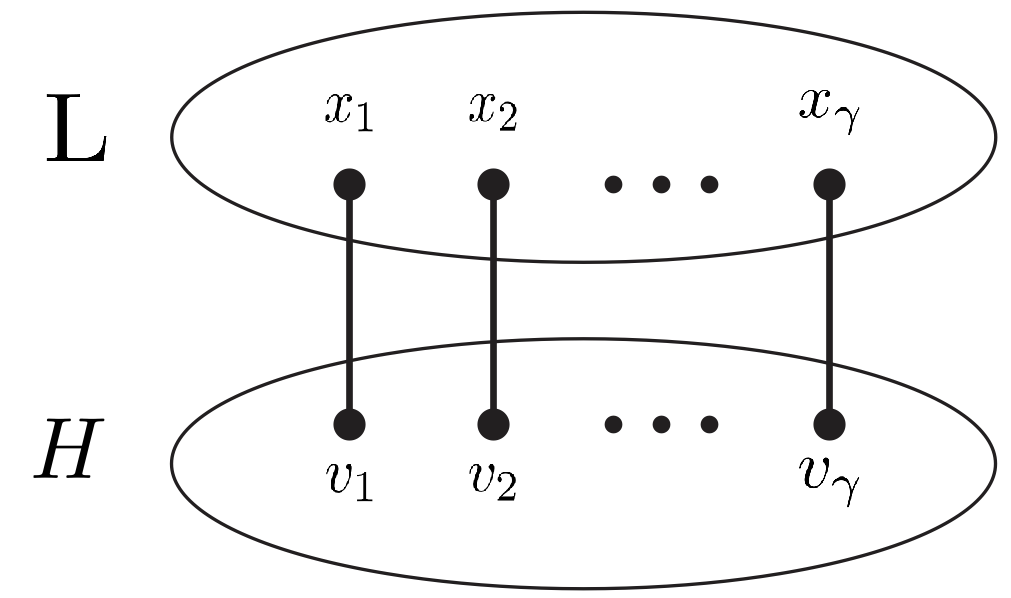}
\caption{}\label{tatabogh}
\end{figure}\\
Suppose that $B$ and $C$ are the adjacency matrices of $H$ and $L$, respectively. So the adjacency matrix of the graph shown in Figure \ref{tatabogh} is as follows:
\begin{center}
$$M= \begin{pmatrix}
B&I\\
I&C
\end{pmatrix}.$$
\end{center}
Next, if we multiply the second row block of $M-\lambda I$ by $-(B-\lambda I)$ and add it to the first row block of $M-\lambda I$, we obtain the following matrix:
\begin{center}
$$M'=\begin{pmatrix}
0&I-(B-\lambda I)(C-\lambda I)\\
I&C-\lambda I
\end{pmatrix}.$$
\end{center}
Note that $rank(A-\lambda I)=\gamma$ implies that $(B-\lambda I)(C-\lambda I)=I$. Note that $B$ and $C$ are $(0,1)$-matrices. Suppose that $B= [b_{ij}]$ and $C=[c_{ij}]$, and consider the $(1,1)$-entry of $(B-\lambda I)(C-\lambda I)=I$, that is,
$\lambda^2+ \sum_{i=1}^{\gamma} b_{1i}c_{i1}=1.$
Thus $\lambda^2 \in \{0,1\}$ and so $\lambda \in \{-1,0,1\}$.
\vspace{0.5\baselineskip}\\
First, suppose that $\lambda=-1$. If $\gamma\geq 2$, since $H$ is connected, there exist $p,q$ such that $1\leq p,q \leq \gamma$ and $b_{pq}=1$. Calculating the $(p,q)$-entry of $(B+ I)(C+ I)=I$ we obtain\\
\hphantom{$(10)$} $ \hspace{32mm}  b_{pq}+ c_{pq} +  \displaystyle\sum_{\substack{i=1\\i\neq p,q}}^{\gamma} b_{pi}c_{iq}=0.$ \\
Since the left side is positive, we get a contradiction. Thus $\gamma=1$ and $m_G(\lambda)=n-1$. Hence $G$ has exactly two distinct eigenvalues and so $G=K_n$.
\vspace{0.5\baselineskip}\\
Now, suppose that $\lambda=0$. This implies that $BC=I$. If $v_1v_i,\, v_1v_j \in E(G)$ and $1\leq i<j\leq \gamma$, then looking at the first row of $BC=I$, one can see that the $i$th or $j$th row of $C$ is zero. Therefore $C$ is singular, a contradiction. Therefore $d_H(v_1)\leq 1$. Since $H$ is connected, $d_H(v_1)=1$. Similarly, for $t=2,\ldots,\gamma$, $d_H(v_t)=1$. Thus $\gamma=2$ and $m_G(0)=n-2$. By \cite[Theorem 2]{chang}, it follows that $G=K_{r,s}$ for some integers $r$ and $s$. Note that $\gamma(K_{1,n-1})=1$, so it follows that $r,s\geq2$.  
\vspace{0.5\baselineskip}\\
Now, suppose that $\lambda=1$ and $\gamma>1$. This implies $(B-I)(C-I)=I$ and so $BC=B+C$. Since $C-I$ is invertible, $L$ is a star complement for $G$ corresponding to $\lambda = 1$. By Lemma \ref{dom}$(ii)$, both $V(H)$ and $V(L)$ are location-dominating sets for $G$. So if $u\in V(G)\setminus (V(H)\cup V(L))$, then $\left| N_H(u)\right|,\,\left| N_L(u)\right| \geq 2$. Note that if $v_iv_j\in E(G)$ for some $1\leq i<j\leq \gamma$, then by calculating the $(i,i)$-entry of $(B-I)(C-I)=I$, one can see that $x_ix_j\notin E(G)$.
 Since $H$ is connected, $H$ has at least one edge. Without loss of generality, we can assume that $v_1v_2 \in E(G)$. Hence $b_{12}=1$ and so $(BC)_{12}=(B+C)_{12}=1$. Therefore there exists a unique integer $l$, $3\leq l\leq \gamma$ such that $v_lv_1,\, x_lx_2 \in E(G)$. Thus $d_H(v_1)\geq 2$. Similarly, for $t=2,\ldots,\gamma$, one can see that $d_H(v_t)\geq 2$.
 Note that the set $\{v_1,\, x_2,\, \ldots, \, x_{\gamma}\}$ is a dominating set for $G$. But the set $\{v_1,\, x_3,\, \ldots, \, x_{\gamma}\}$ is not a dominating set for $G$. Hence by Lemma \ref{dom}$(ii)$, there is a vertex $u$ such that $N_L(u)=\{x_1,\,x_2\}$. Also, the set $\{u,\, v_3,\ldots,\,v_{\gamma}\}$ is not a dominating set for $G$. Similarly, there is a vertex $v$ such that $N_H(v)=\{v_1,\,v_2\}$ and by Lemma \ref{dom}$(ii)$, there is no $w\in V(G)\setminus V(H)$ and $w\neq v$ such that $N_H(w)=\{v_1,v_2\}$. We need to consider two cases: 

\vspace{0.3\baselineskip}

\textbf{ Case 1}. Suppose $vx_1,vx_2\notin E(G)$, and let $D$ be the adjacency matrix of the induced subgraph on $V(H)\cup\{x_1,\, x_2,\, v\}$, see also Figure \ref{case12}$(a)$:
\begin{figure}[h]
\centering
\includegraphics[width=80mm]{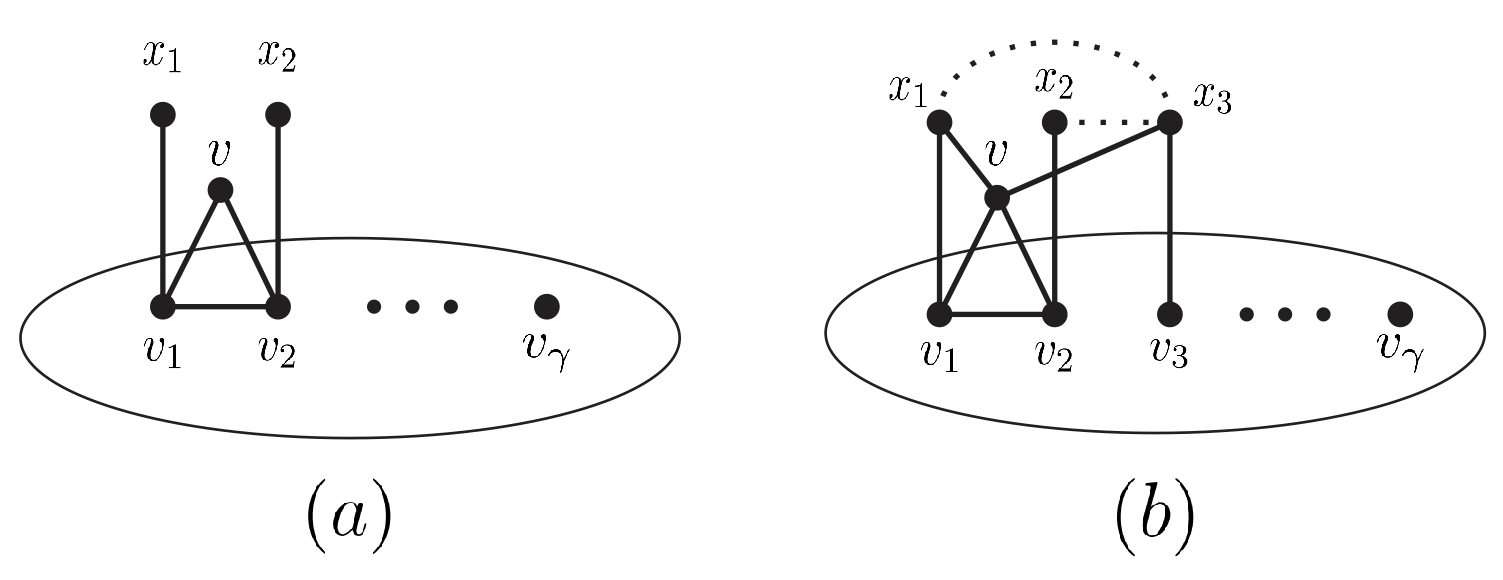}
\caption{}\label{case12}
\end{figure}
\begin{center}
$$ \resizebox{.05\textwidth}{!}{  D  }= \mleft(
\begin{array}{ccccc|ccc}
&&&&&1&0&1\\
&&&&&0&1&1\\
&& \resizebox{.05\textwidth}{!}{  B  }&&&0&0&0\\
&&&&& \vdots&\vdots&\vdots\\
&&&&&0&0&0\\
\hline
1&0&0&\dots &0&0&0&0\\
0&1&0&\dots &0&0&0&0\\
1&1&0&\dots &0&0&0&0
\end{array}
\mright).
$$
\end{center}  
Observe that the last three rows and columns of $D$ correspond to the vertices $x_1,\,x_2$ and $v$, respectively. By the Interlacing Theorem, we know that the multiplicity of each eigenvalue $\lambda$ decreases at most one after removing a vertex. So $m_D(1)=3$. Hence, there is a non-zero eigenvector $z=[z_1,\ldots, z_{\gamma+3}]^\top$ corresponding to the eigenvalue $1$ such that $z_1=z_2=0$. Calculating the last three components of $z$, we see that $z_{\gamma+1}=z_{\gamma+2}=z_{\gamma+3}=0$. Thus, removing the last three components of $z$, we obtained a non-zero eigenvector corresponding to the eigenvalue $1$ for $B$, a contradiction.

\vspace{0.3\baselineskip}

\textbf {Case 2}. Without loss of generality, suppose that $vx_1\in E(G)$. Since $N_L(v)\geq 2$, then there exists at least a vertex $x_i$ such that $vx_i\in E(G)$. If $vx_2\in E(G)$, then by Lemma \ref{dom}$(ii)$ we know that $V(H)$ is a location-dominating set for $G$, so the set $\{v,\,v_3,\ldots,\,v_{\gamma}\}$ is a dominating set, a contradiction. Thus $vx_2\notin E(G)$. Without loss of generality, assume that $vx_3\in E(G)$. Let $D$ be the adjacency matrix of the induced subgraph on $V(H)\cup\{x_1,\, x_2,\, x_3,\, v\}$, see also Figure \ref{case12}$(b)$:
\begin{center}
$$ \resizebox{.05\textwidth}{!}{  D  }= \mleft(
\begin{array}{cccccc|cccc}
&&&&&&1&0&0&1\\
&&&&&&0&1&0&1\\
&&&&&&0&0&1&0\\
&&& \resizebox{.05\textwidth}{!}{  B  }&&&0&0&0&0\\
&&&&&& \vdots&\vdots&\vdots&\vdots\\
&&&&&&0&0&0&0\\
\hline
1&0&0&0&\dots &0&0&0&a&1\\
0&1&0&0&\dots &0&0&0&b&0\\
0&0&1&0&\dots &0&a&b&0&1\\
1&1&0&0&\dots &0&1&0&1&0
\end{array}
\mright),
$$
\end{center} 
where the last four rows and columns of $D$ are corresponding to $x_1,\,x_2,\,x_3$ and $v$, respectively and $a,b\in \{0,\,1\}$. Again, using  Interlacing, it follows that $m_D(1)=4$. Thus there is a non-zero eigenvector $z=[z_1,\ldots ,z_{\gamma+4}]^\top$ corresponding to the eigenvalue $1$ such that $z_1=z_2=z_3=0$. Calculating the last four components of $z$ we obtain that the following equations must hold:
\vspace{0.3\baselineskip}

$\begin{cases}
z_{\gamma+1}=az_{\gamma+3}+z_{\gamma+4} \\
z_{\gamma+2}=bz_{\gamma+3}\\
z_{\gamma+3}=az_{\gamma+1}+bz_{\gamma+2}+z_{\gamma+4}\\
z_{\gamma+4}=z_{\gamma+1}+z_{\gamma+3}.
\end{cases}$ 
\vspace{0.3\baselineskip}

From a straightforward calculation, we obtain that $z_{\gamma+1}=z_{\gamma+2}=z_{\gamma+3}=z_{\gamma+4}=0$. Thus, removing the last four components of $z$ we obtained a non-zero eigenvector corresponding to the eigenvalue $1$ for $B$, a contradiction.\\
Therefore $\gamma=1$ and $m_G(1)=n-1$. So $n=2$ and $G= K_2$. 
\end{proof}

As a direct consequence of Theorem \ref{thm:first} we obtain the following result, which was shown for general graphs by Wang, Yan, Fang, Geng and Tian \cite[Theorem 4.5]{wyfgt}.

\begin{cor}
Let $G$ be a connected regular graph of order $n$, and let $\lambda$ be a Laplacian eigenvalue of $G$ with multiplicity $m_G(\lambda)$. Then $\gamma(G) \leq n-m_G(\lambda)$.
\end{cor}

Note that Theorem \ref{thm:first} characterizes the graphs $G$ in which equality occurs for the new spectral bound of $\gamma(G)$. Our second main result provides a complete characterization of the equality case for the bound of $\gamma_t(G)$ when $ \lambda\notin \mathbb{Q}$ or $\lambda=0$. If $\lambda \in \mathbb{Q}$, then obviously $\lambda$ must be an integer. In this case, $\lambda \in \{\ldots, -3,-2,-1,0,1\}$.

\begin{thm}\label{thm:second}
Suppose that $G$ is a connected graph of order $n$ and $\lambda$ is an eigenvalue of $G$ with multiplicity $m_G(\lambda)$. Then the following hold:
\begin{description}
\item[$(i)$]  If $\lambda\notin \mathbb{Q}$, then  $\gamma_t(G)=n-m_G(\lambda)$ if and only if $(\lambda, G)=(\frac{-1\pm \sqrt{5}}{2}, C_5)$ or $(\lambda, G)=(\pm \sqrt{2}, K_{1,2})$.
\item[$(ii)$] If $\lambda\in \mathbb{Q}$, then $\lambda$ is an integer and is at most $1$.
\end{description}
Moreover, if $\lambda=0$, then $\gamma_t(G)=n-m_G(\lambda)$ if and only if $G= K_{r,s}$ for some positive integers $r$ and $s$.
\end{thm}

\begin{proof}
For the sake of notation simplicity, we denote $\gamma_t(G)=\gamma$. 
From the first part of the proof of Theorem \ref{thm:first}, we know that the star complement $V(H)$ of the eigenvalue $\lambda$ is a total dominating set with $|V(H)|=\gamma_t(G)$ and the induced subgraph $H$ is connected. Therefore, for every $i=1,\ldots,\gamma$, the set $V(H)\setminus \{v_i\}$ is not a total dominating set. So there is at least one vertex $x_i$ such that $N_H(x_i)=\{v_i\}$. Assume that, for an integer $t$ with $0\leq t\leq \frac{\gamma}{2}$ and  $i=t+1,\ldots ,\gamma$, there is at least one vertex $x_i\notin V(H)$ such that $N_H(x_i)=\{v_i\}$. Assume also  that for $i=1,\ldots, t$, there is no vertex $u\notin V(H)$, $N_H(u)=\{v_i\}$, but there is at least one vertex $v_{j_i}\in V(H)$ such that $N_H(v_{j_i})=\{v_i\}$. Without loss of generality, suppose that $j_i=t+i$. Let $V_1=\{v_1,\ldots,v_t\}$, $V_2=\{v_{t+1},\ldots, v_{2t}\}$, $V_3=\{v_{2t+1},\ldots,v_{\gamma}\}$, $V_4=\{x_{t+1},\ldots, v_{2t}\}$ and $V_5=\{x_{2t+1},\ldots,x_{\gamma}\}$, see Figure \ref{fiv}.
\begin{figure}[h!]
\centering
\includegraphics[width=80mm]{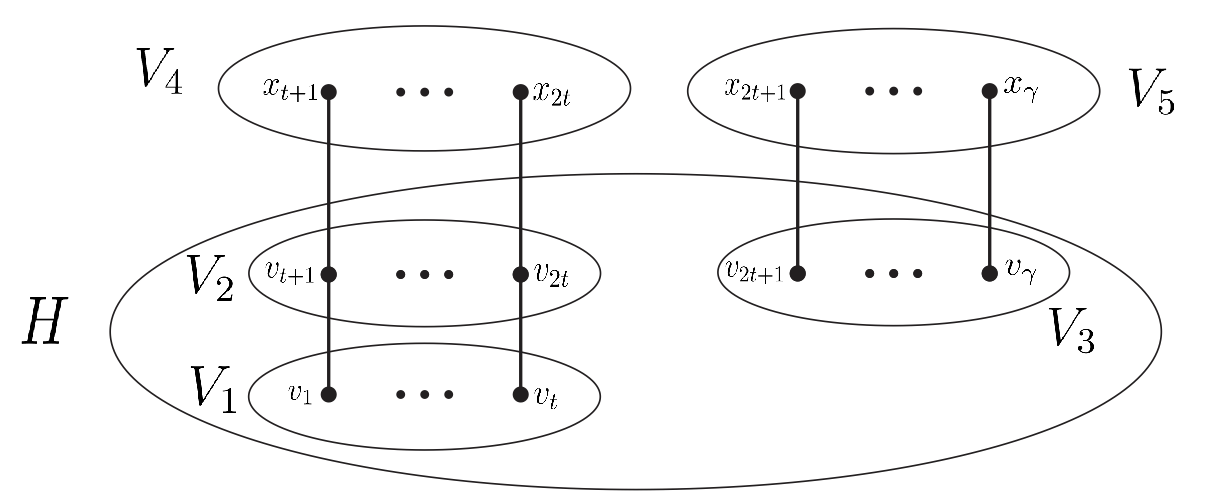}
\caption{}\label{fiv}
\end{figure}
\\
Let $P$ be the adjacency matrix of the graph shown in Figure \ref{fiv}, that is\\

\begin{center}
$$ P-\lambda I = \mleft(
\begin{array}{ccccc}
B_1-\lambda I&I&B_{13}&0&0\\
I&-\lambda I&0&I&0\\
B_{13}^T&0&B_3-\lambda I&0&I\\
0&I&0&B_4-\lambda I&B_{45}\\
0&0&I&B_{45}^T&B_5-\lambda I
\end{array}
\mright),
$$
\end{center}
where $B_1,B_3,B_4$ and $B_5$ are the adjacency matrices of the induced subgraphs on $V_1,V_3$ and $V_5$, respectively, and $B_{13}$ and $B_{45}$ are the incidence matrices of $V_1,V_3$ and $V_4,V_5$, respectively. Since $V(H)$ is a star complement of $G$ corresponding to the eigenvalue $\lambda$, the principal submatrix of $P-\lambda I$ of order $3$ has full rank. So the forth and fifth rows of $P-\lambda I$ are a linear combination of the first three rows of $P-\lambda I$, which implies that the following holds:\\
$(1) \hspace{3.1em} \lambda B_4B_1+(1-\lambda^2)B_1+(1-\lambda^2)B_4+\lambda(\lambda^2-2)I+B_{45}B_{13}^T=0$,\\
$(2) \hspace{3.1em} (1-\lambda^2)B_{13}+\lambda B_4B_{13}+B_{45}B_3-\lambda B_{45}=0$,\\
$(3) \hspace{3.1em} (1-\lambda^2)B_{45}+\lambda B_1B_{45}+B_{13}B_5-\lambda B_{13}=0$,\\
$(4) \hspace{3.1em} \lambda B_{45}^TB_{13}+B_5B_3-\lambda B_5-\lambda B_3+(\lambda^2-1)I=0$.\\

\noindent First, suppose that $\lambda=0$. We can rewrite equations $(1)$ and $(2)$ as follows: \\
$\hspace{38mm} B_1+B_4+B_{45}B_{13}^T=0$,\\
$\hspace{42mm} B_{13}+B_{45}B_3=0$.\\
Since all  matrices are $(0,1)$-matrices, we know that $B_1=B_4=B_{13}=0$. But unless $\gamma-2t=0$ or $t=0$, this contradicts the connectivity of $H$. If $\gamma-2t=0$, then $t=1$. Thus $\gamma=2$ and by \cite[Theorem 2]{chang}, it follows that $G= K_{r,s}$, for some positive integers $r,s$. But since there is no vertex $u\notin V(H)$ such that $N_H(u)=\{v_1\}$, at least one of $r$ or $s$ is $1$. Therefore $G=K_{1,n-1}$. If $t=0$, then $\gamma(G)=n-m_G(\lambda)$. Since $\gamma_t(K_2)=\gamma_t(K_n)=2$ and $\gamma_t(K_1)=\infty$, by Theorem \ref{thm:first}, $(\lambda,G)=(0,K_{r,s})$, for some integers $r,s\geq 2$.

\vspace{0.3\baselineskip}

\noindent Now, suppose that $\lambda \notin \mathbb{Q}$.

\vspace{0.3\baselineskip}

\noindent If $\gamma-2t=0$, then $B_3=B_5=B_{13}=B_{45}=0$. Using Equation $(1)$ we obtain $\lambda B_4B_1+(1-\lambda^2)B_1+(1-\lambda^2)B_4+\lambda(\lambda^2-2)I=0$. By looking at the $(1,1)$-entry of this equation we note that $2-\lambda^2$ must be an integer. So $\lambda \in \{0,\pm 1,\pm\sqrt{2}\}$. Thus $\lambda= \pm \sqrt{2}$. Then $B_1B_4=0$ and so $B_1=B_4=0$, because $B_1$ and $B_4$ are non-negative matrices. Since $H$ is connected, $t=1$ and so $\gamma=2$ and $m_G(\lambda)=n-2\geq 1$. Since $\lambda$ is not an integer, there exists another eigenvalue conjugate to $\lambda$ with multiplicity $n-2$. Thus $2(n-2)\leq n$ and so $3\leq n\leq 4$.  If $n=4$, then $G$ has only two distinct eigenvalues. Therefore it should be a complete graph and so $\lambda$ cannot be its eigenvalue, a contradiction. So $n=3$ and $G=K_{1,2}$.

\vspace{0.3\baselineskip}

\noindent Now, suppose that $t,\gamma-2t\geq 1 $. By looking at the $(1,1)$-entries of equations (1) and (4) we obtain:\\
$(5)  \hspace{3.1em} a\lambda+\lambda(\lambda^2-2)+b=0$,\\
$(6) \hspace{3.1em} c\lambda+d +\lambda^2-1=0$,\\
where $a$, $b$, $c$ and $d$ are the $(1,1)$-entries of $B_4B_1$, $B_{45}B_{13}^T$, $B_{45}^TB_{13}$ and $B_5B_3$, respectively. Using now Equation $(6)$, $\lambda=\dfrac{-c\pm \sqrt{c^2-4d+4}}{2} $ and so $c^2-4d+4\geq 0$. From Equation $(6)$ we also see that $\lambda(\lambda^2-2)=\lambda(c^2-d-1)+cd-c$, and from Equation $(5)$, $\lambda=\dfrac{c-cd-b}{c^2+a-d-1}$. Since $\lambda$ is not a rational number, $c^2+a-d-1=0$ and so $a=d+1-c^2\geq 0$. Therefore $c^2-d-1\leq c^2-4d+4$ and so $d\in \{0,1\}$. If $d=1$, then $\lambda$ is an integer, a contradiction. Therefore $d=0$ and so $c\in \{0,1\}$. Note that $c\neq 0$, since $\lambda$ is not integer. Thus $\lambda=\dfrac{-1\pm\sqrt{5}}{2}$.\\
Now, from Equation ($6$), one can see that $\lambda^2-1=-\lambda$. So equations ($1$) to ($4$) can be rewritten as follows:\\
$(1') \hspace{3.1em} \lambda B_4B_1+\lambda B_1+\lambda B_4-I+B_{45}B_{13}^T=0$,\\
$(2')  \hspace{3.1em} \lambda B_{13}+\lambda B_4B_{13}+B_{45}B_3-\lambda B_{45}=0$,\\
$(3') \hspace{3.1em} \lambda B_{45}+\lambda B_1B_{45}+B_{13}B_5-\lambda B_{13}=0$,\\
$(4') \hspace{3.1em} \lambda B_{45}^TB_{13}+B_5B_3-\lambda B_5-\lambda B_3-\lambda I=0$.

\vspace{0.3\baselineskip}

\noindent Since all matrices are $(0,1)$-matrices, by $(1')$, it is obvious that $B_1=B_4=0$ and so $B_{45}B_{13}^T=I$. Summing up equations $(2')$ and $(3')$ we obtain $B_4B_{13}=B_{45}B_3=B_1B_{45}=B_{13}B_5=0$, and hence $B_{13}=B_{45}$. Therefore $B_{13}B_{13}^T=I$ and so for every $v_i\in V_1$, $d_{V_3}(v_i)=1$. By the  Equation $(4')$, $B_5B_3=\lambda(-B_{45}^TB_{13}+B_5+B_3+I)$. Therefore $B_5B_3=0$ and so we can rewrite $(4')$ as $B_{13}^TB_{13}=B_5+B_3+I$. By multiplying this equation by $B_3$ on the right side, we get $B_3^2+B_3=0$ and so $B_3=0$. Similarly, $B_5=0$ and so $B_{13}^TB_{13}=I$. Therefore for every $v_i\in V_3$, $d_{V_1}(v_i)=1$. Since $H$ is connected, $\gamma-2t=t$ and $t=1$. Therefore $\gamma=3$ and $n\geq 5$. So $m_G(\lambda)=n-3$. Since $\lambda$ is not an integer, there exists a conjugate of $\lambda$ with multiplicity $n-3$. Therefore $2(n-3)\leq n$, which implies $5\leq n\leq 6$. If $n=6$, then since $G$ has only two eigenvalues it should be a complete graph, a contradiction. Thus, $n=5$ and since each set $V_i$, for $i=1,\ldots , 5$, has at least one vertex, it follows that $|V_i|=1$. Suppose that $v_i\in V_i$, for $i=1,\ldots ,5$. Since $H$ is connected, by the definition of sets $V_4$ and $ V_5$, the path $v_4v_2v_1v_3v_5$ is a subgraph of $G$. But the path $P_5$ has five distinct eigenvalues and so $v_4$ and $v_5$ are adjacent. Hence $G=C_5$.

\vspace{0.3\baselineskip}

\noindent Finally, if $\lambda \in \mathbb{Q} $, then it is clear that it must be an integer. Also, since all matrices are $(0,1)$-matrices, by looking at the $(1,1)$-entry of Equation $(4)$, it is easy to see that $\lambda\leq 1$, which implies that $\lambda$ must be an integer of value at most $1$. 

\vspace{0.3\baselineskip}

\noindent 
\end{proof}

\section{Star sets and $p$-domination number}
In this section we show several bounds for the $p$-domination number of a graph using the multiplicity of any eigenvalue of $G$. In particular, first we prove a new relation between star complements and $p$-dominating sets (Theorem \ref{gene}), which extends Lemma \ref{dom}. As an application of it, we obtain a bound for the $p$-dominating number in terms of the multiplicity of any  adjacency eigenvalue (Corollary \ref{coro:starsetspdomination}). 
In the second part of this section, we show that the results of  Wang, Yan, Fang, Geng and Tian \cite[Section 4]{wyfgt}, who used the Laplacian eigenvalues to provide a bound for the domination number, also hold if one considers the adjacency matrix and the $p$-domination number (Theorem \ref{tok}).

\begin{thm}\label{gene}
Let $G$ be a graph of order $n$, with minimum degree $\delta(G)$ and having $s$ distinct eigenvalues, and let $p$ be a positive integer such that $p\leq \delta(G)$. If $s\geq n-\delta(G)+p$, then for every eigenvalue of $G$, its star complement set is a $p$-dominating set.
\end{thm}

\begin{proof}
Consider a graph $G$ having distinct eigenvalues  $\{\lambda_1, \ldots, \lambda_s\}$ with multiplicities $m_1 \leq \cdots \leq  m_s$, respectively. Let $X_i$ be a star set for $\lambda_i$, for $i = 1, \ldots, s$. Note that
\begin{center}
 $n=\displaystyle\sum_{i=1}^s m_i\geq  s-1+m_s\geq n-\delta(G)+p-1+m_s$.
\end{center}
 So $m_s\leq \delta(G)-p+1$. Therefore $|X_i|=m_i\leq \delta(G)-p+1$, for $i = 1, \ldots, s$. Hence every vertex of $X_i$ is adjacent
to at least $\delta(G) - (\delta(G) - p) = p$ vertices in $G\setminus X_i$. Hence $G\setminus X_i$ is a $p$-dominating set.
\end{proof}

Note that for $p=1$, Theorem \ref{gene} gives Lemma \ref{dom}.
As a corollary of Theorem \ref{gene} we obtain a new bound for the $p$-dominating number in terms of the multiplicity of any adjacency eigenvalue.

\begin{cor}\label{coro:starsetspdomination}
Let $G$ be a connected graph of order $n$ with $s$ distinct eigenvalues and let $p$ be a positive integer such that $p\leq \delta(G)$. Suppose that $\lambda$ is an adjacency eigenvalue of $G$ with multiplicity $m_G(\lambda)$. If $s\geq n-\delta(G) +p$, then $\gamma_p(G) \leq n-m_G(\lambda)$.
\end{cor}

\begin{proof}
Assume that $X$ is the star set for eigenvalue $\lambda$ of $G$. By Theorem \ref{gene}, we know that the star complement $\overline{X}$ is a $p$-dominating set. So it follows $\gamma_p(G) \leq |\overline{X}| =n-m_G(\lambda)$.
\end{proof}

Finally, we show that an analogous bound also holds for $\gamma_p$ if one uses the adjacency or Laplacian eigenvalues. For a graph $G$ of order $n$, we call $L(G) = D(G) - A(G)$ the {\it Laplacian matrix} of $G$, where $D(G)$ is the $n \times n$ diagonal matrix with $d_{ii}=d(v_i)$. The eigenvalues of $L(G)$ are called the {\it Laplacian eigenvalues}. Let $G$ be a graph of order $n$. For a subset $S$ of $V(G)$, let $Z(S)$ be the subset of $\mathbb{R}^n$ consisting of all vectors $\alpha$ for which $\alpha_v=0$, for all $v\in S$. Then $Z(S)$ is a subspace of $\mathbb{R}^n$ of dimension $n -|S|$.  
For an eigenvalue $\lambda$ of $G$, we call a subset $S$ of $V(G)$ a $\lambda$-\emph{annihilator} of $G$ if $Z(S) \cap Ker(L(G) -\lambda I_n) =0$, where $0$ refers to the vector subspace of $\mathbb{R}^n$ consisting of only the zero vector. We note that the proof of \cite[Lemma 4.3]{wyfgt} also holds for the adjacency eigenvalues:
\begin{lem}\label{ala}
Let $G$ be a graph with $\lambda$ as an adjacency or Laplacian eigenvalue. If $S$ is a $\lambda$-annihilator of $G$, then $m_G(\lambda) \leq |S|$.
\end{lem}

The previous bound on the cardinality of a minimum $\lambda$-annihilator of $G$ can be used to establish an upper bound for $\gamma_p$.

\begin{thm}\label{tok}
Let $G$ be a graph of order $n$ and with $p$-domination number $\gamma_p >p$. If $G$ has a minimum $p$-dominating set in which every of its $p$ vertices has an external private neighbor, then for any Laplacian or adjacency eigenvalue $\lambda$ it holds $\gamma_p \leq n -m_G(\lambda)$.
\end{thm}
\begin{proof}
Assume that $\lambda$ is an adjacency eigenvalue of $G$. The proof for the Laplacian eigenvalues is analogous. 
Suppose that $X$ is the minimum $p$-dominating set in which every of its $p$ members has an external private neighbor. Put $S= G\setminus X$. By Lemma \ref{ala} it is enough to prove that $S$ is a $\lambda$-annihilator of $G$. Suppose $\alpha\in Z(S) \cap Ker(L(G) -\lambda I_n)$. Then $\alpha_w=0$, for every $w\in S$ and $A(G)\alpha =\lambda \alpha$. For an arbitrary vertex $v\in X$, we claim that $\alpha_v=0$. For this, suppose that $p$ vertices $v_1,\ldots ,v_p \in X$ have an external private neighbor $u$. By calculating the row corresponding to $u$ in $ A(G)\alpha =\lambda \alpha$, we obtain $\alpha_{v_1}+\cdots +\alpha_{v_p}=0$. Since $\gamma_p >p$, suppose that $ v_2,\ldots ,v_p,v_{p+1} \in X$ have an external private neighbor $u'$. Again, by calculating the row corresponding to $u'$ in $ A(G)\alpha =\lambda \alpha$, we obtain $\alpha_{v_2}+\cdots +\alpha_{v_p}+\alpha_{v_{p+1}}=0$. So $\alpha_{v_{p+1}}=\alpha_{v_1}$. Similarly, $\alpha_{v_{p+1}}=\alpha_{v_i}$, for $i=2,\ldots ,p$. Therefore the claim is proved and $m_G(\lambda) \leq n -\gamma_p$.
\end{proof}

\section*{Acknowledgements}
A. Abiad is partially funded by the Fonds Wetenschappelijk Onderzoek (FWO), grant 1285921N.\\
 S. Akbari is partially funded by the Iran National Science Foundation (INSF), grant 96004167.



\end{document}